\documentclass{amsart}
\usepackage{amsmath}
\usepackage{amscd}
\usepackage{amssymb}
\usepackage{graphicx}
\usepackage{graphics}
\usepackage{latexsym}

\input xy
\xyoption{all}

\newtheorem{theorem}{Theorem}[section]
\newtheorem{lemma}[theorem]{Lemma}
\newtheorem{proposition}[theorem]{Proposition}

\theoremstyle{definition}
\newtheorem{definition}[theorem]{Definition}

\newcommand{\zn}{{\mathbb N}}
\newcommand{\bop}{\bigoplus}
\newcommand{\rf}{{\mathcal R}{\mathcal F}}
\newcommand{\zp}{{\mathbb P}}
\newcommand{\tor}{{\mathbf{Tor}}}

\begin{document}

\title{A Refinement of Multi-dimensional Persistence}
\author{Kevin P.~Knudson}\thanks{Partially supported by the National Security Agency and DARPA}
\email{knudson@math.msstate.edu }
\address{Department  of Mathematics and Statistics, Mississippi State University, Mississippi State, MS 39762}

 \subjclass{Primary: . Secondary:}
\date{\today}

\begin{abstract}We study the multi-dimensional persistence of Carlsson and Zomorodian \cite{cz} and obtain a finer classification based upon
the higher tor-modules of a persistence module.  We propose a variety structure on the set of isomorphism classes of these modules, and present
several examples.  We also provide a geometric interpretation for the higher tor-modules of homology modules of multi-filtered simplicial
complexes.
\end{abstract}

\maketitle

 \section{Introduction}\label{intro}
 Persistent homology has become a popular tool in the study of point cloud data sets.  Given such a set $X$, one may attempt to approximate the
 topology of $X$ by first placing $\varepsilon$-balls around each point (call the union of the balls $X_\varepsilon$), and then allowing
 $\varepsilon$ to grow.  This yields a nested sequence of spaces $X_\varepsilon\subset X_{\varepsilon'}$, $\varepsilon<\varepsilon'$, and one
 may compute the homology of these spaces.  For $\varepsilon$ small, not much happens since $X_\varepsilon$ is simply a disjoint union of balls,
 but as $\varepsilon$ increases the balls begin to overlap and nontrivial cycles may appear.  One may then measure how long such cycles
 ``persist"; that is, a cycle may appear in $X_\varepsilon$ and be filled in by a boundary in some $X_{\varepsilon'}$,
 $\varepsilon'>\varepsilon$.  If the difference $\varepsilon'-\varepsilon$ is large relative to $\varepsilon$ then one may deduce that the cycle
 is a real topological feature of the set $X$.  For interesting applications of these ideas see, for example,
 \cite{cidz},\cite{czcg},\cite{dgm},\cite{ehz}.

 The abstraction of this idea is the notion of a filtered space.  Given a space $X$, which in this paper will always be a simplicial complex,
 we take an increasing sequence of subcomplexes $$\emptyset=X_{-1}\subset X_0\subset X_1\subset \cdots \subset X_r=X.$$ Let $k$ be a field.
 We then obtain, for $i\ge 0$, a sequence of $k$-vector spaces $$0\to H_i(X_0;k)\to H_i(X_1;k)\to\cdots\to H_i(X;k),$$ and we may observe how
 long a homology class persists in this sequence.  This is encapsulated neatly by Carlsson and Zomorodian \cite{zc} in the following way.  Let
 $M=\bop_{j\ge 0} H_i(X_j;k)$.  This is a module over the polynomial ring $k[x]$ where the action of $x$ on each $H_i(X_j;k)$ is given by the
 map $H_i(X_j;k)\to H_i(X_{j+1};k)$.  The classification of modules over $k[x]$ implies that $M\cong \bop_{j=1}^r x^{\alpha_j}k[x]\oplus
 \bop_{m=1}^p x^{\beta_m}/x^{s_m}$.  In turn, this yields a {\em barcode} for $M$.  This is a set of intervals $[\alpha_j,\infty)$,
 $[\beta_m,\beta_m+s_m]$ that show how long homology classes persist, the infinite intervals corresponding to classes that live in $H_i(X;k)$.

 In applications, however, one may need to consider multiple filtration directions.  For example, the data set in question may have a natural
 filtration of its own (e.g., by density), and then we obtain another filtration direction by growing $\varepsilon$-balls.  These {\em
 multifiltrations} are much more complicated, and the complete classification of \cite{zc} has no analogue.  Indeed, in \cite{cz}, the authors
 show that for multifiltrations there is no complete discrete invariant analogous to the barcode.  There are some discrete invariants, but there
 is also a continuous piece obtained as a quotient of an algebraic variety, $\rf(\xi_0,\xi_1)$.  This is summarized in Section \ref{variety}
 below.

 The main idea in \cite{cz} is to considered spaces $X$ filtered by $X_v\subset X$ for $v\in\zn^n$.  For a fixed $i\ge 0$, one then obtains a
 $k[x_1,\dots ,x_n]$-module $M=\bop_{v\in\zn^n} H_i(X_v;k)$ just as in the $n=1$ case.  Modules over $A_n=k[x_1,\dots ,x_n]$, $n\ge 2$, do not admit
 a neat classification, however, and that is where the trouble lies.  The authors consider two multisets $\xi_0$ and $\xi_1$ which indicate the
 degrees in $\zn^n$ where homology classes are born and where they die, respectively.  The problem is that there may be many (even uncountably
 many) nonisomorphic modules with the same $\xi_0$ and $\xi_1$.  These are parametrized bye the quotient of the variety $\rf(\xi_0,\xi_1)$ by
 the action of an algebraic group.  In the case $n=1$, this quotient space is always finite (see Theorem \ref{onevar} below), as one would
 expect given the classification of these modules discussed above.

 The multisets $\xi_0$ and $\xi_1$ consist of the elements in $\zn^n$ where the module $M$ has generators and relations, respectively.  These
 are obtained by computing the modules $\text{Tor}_0^{A_n}(M,k)$ and $\text{Tor}_1^{A_n}(M,k)$.  When $n=1$, these are the only nontrivial Tor
 groups, but for $n\ge 2$, there may be more.  These higher Tor modules are the main objects of study in this paper.  For $i\ge 0$, let
 $\xi_i(M)$ be the multiset of elements in $\zn^n$ where generators of $\text{Tor}_i^{A_n}(M,k)$ occur.  Hilbert's Syzygy Theorem implies that
 $\xi_i=\emptyset$ for $i>n$ and so we obtain a finite family of discrete invariants, $\xi_0(M),\xi_1(M),\dots ,\xi_n(M)$.  Using these, we may
 partition the set $\rf(\xi_0,\xi_1)$ of \cite{cz} into subsets $\rf(\xi_2,\dots ,\xi_n)$ consisting of those modules $M$ having
 $\xi_i(M)=\xi_i$.  Let $F(\xi_0)$ be the free $A_n$-module with basis $\xi_0$ and let $GL(F(\xi_0))$ be the group of degree-preserving
 automorphisms of $F(\xi_0)$.  This group acts on the various $\rf(\xi_2,\dots ,\xi_n)$ and we have the following result.

 \medskip

 \noindent {\bf Theorem 4.3} {\em There is a projective variety $Y_{\xi_2,\dots ,\xi_n}$ and a map
 $$\varphi:GL(F(\xi_0))\backslash\rf(\xi_2,\dots ,\xi_n)\to Y_{\xi_2,\dots ,\xi_n}.$$}

 Often, the map $\varphi$ is injective and we may use it to give the quotient set  the structure
 of an algebraic variety.  In turn, this yields a variety structure on the full quotient
  $GL(F(\xi_0))\backslash\rf(\xi_0,\xi_1)$ by taking the disjoint union
 over the possible $\xi_2,\dots ,\xi_n$.  Morally, this is what one wants.  However, this is {\em not} the quotient space obtained by viewing
 $\rf(\xi_0,\xi_1)$ as a variety and then taking the quotient by $GL(F(\xi_0))$.  The difference in our approach is that we have lost
 information about certain degeneracies among the elements of $\rf(\xi_0,\xi_1)$ at the expense of gaining a variety structure on the quotient.
 An example of this is given in Section \ref{czex}.

 It may happen, however, that some $\varphi$ is not injective.  This may occur, for example, if there are generators for $M$ that are not
 co-located and the relations lie in unfortunate locations.  We provide a remedy for this in Section \ref{different}.  For an example, see
 Section \ref{containment}.

 The remainder of the paper explores geometric interpretations of the $\xi_i$, $i\ge 2$.  To do this, we back up a step and consider modules of
 chains on a multifiltered space, rather than the individual homology modules.  Let $X_\bullet$ be a complex filtered by $\zn^n$ and for each
 $i\ge 0$, denote by $C_i(X_\bullet)$ the $A_n$-module of $i$-chains on $X_\bullet$: $C_i(X_\bullet) = \bop_{v\in\zn^n} C_i(X_v;k)$.  We then
 have a chain complex $C_\bullet(X_\bullet)$ in the category of graded $A_n$-modules and the associated hypertor modules
 $\tor_j^{A_n}(C_\bullet(X_\bullet),k)$.  By examining the spectral sequences that compute these modules, we obtain a natural map
 $$d^2_{2q}:\text{Tor}_2^{A_n}(H_q(X_\bullet),k)\to\text{Tor}_0^{A_n}(H_{q+1}(X_\bullet),k).$$ This gives us a geometric interpretation of
 $\xi_2(H_q(X_\bullet))$:  elements in $\xi_2$ possibly correspond to locations of generators of $H_{q+1}(X_\bullet)$.  In Theorem
 \ref{tor0tor2}, we describe the kernel and image of this map.

 The higher differentials in this spectral sequence provide a mechanism to relate elements of $\text{Tor}_\ell^{A_n}(H_q(X_\bullet),k)$ to
 elements of $\text{Tor}_0^{A_n}(H_{q+\ell-1}(X_\bullet),k)$.  We shall not investigate these more subtle relationships here.

 In the final section, we use the other spectral sequence to obtain an interpretation of the hypertor modules
 $\tor_j^{A_n}(C_\bullet(X_\bullet),k)$.  If the filtration is such that at most one simplex gets added at a time as we move from one degree to
 another adjacent to it, then we have (Theorem \ref{hypertor})
 $$\tor_\ell^{A_n}(C_\bullet(X_\bullet),k)=\bop_{p+q=\ell}\text{Tor}_q^{A_n}(C_p(X_\bullet),k).$$  Elements on the right hand side may be
 thought of as ``virtual" $(p+q)$-cells that fill in duplicated relations among cells of lower dimension.  We also show that, by dropping the
 grading in the vector spaces $\text{Tor}_q^{A_n}(C_p(X_\bullet),k)$, we may define a boundary operator $\partial:\tor_\ell\to\tor_{\ell-1}$ so
 that the homology of the resulting complex recovers $H_\bullet(X;k)$.  Examples are discussed.

 Finally, we note that there may be a relationship between the $\xi_i$, $i\ge 2$, and the rank invariant of \cite{cz}.  This will be explored
 elsewhere.

 \medskip

 \noindent {\bf Acknowledgements.} Most of the results in this paper were obtained while the author was a member of the Mathematical Sciences
 Research Institute in Berkeley, CA, during the Fall 2006 program on Computational Applications of Algebraic Topology.  The author thanks the
 organizers for putting together an excellent program and MSRI for its hospitality.  The author also thanks Gunnar Carlsson, John Harer, and
 Mark Walker for many helpful conversations about the material in this paper.

\section{Multi-filtered Spaces and Persistence Modules}\label{filtration}
In this section we establish notation and make some definitions.  We keep the notation and terminology of \cite{cz}.  A {\em multiset} is a pair
$(S,\mu)$, where $S$ is a set and $\mu:S\to\zn$ specifies the multiplicity of each element of $S$.  For example, the multiset $\{a,a,a,b,b,c\}$
has $\mu(a)=3$, $\mu(b)=2$ and $\mu(c)=1$; we represent this as $(\{a,b,c\},\mu)$ or as $\{(a,3),(b,2),(c,1)\}$.

Given elements $u,v\in\zn^n$, we say $u\lesssim v$ if $u_i\le v_i$ for each $1\le i\le n$.  If $(S,\mu)$ is a multiset with $S\subseteq \zn^n$,
then $\lesssim$ is a quasi-partial order on $(S,\mu)$.  If $k$ is a field, denote by $k[x_1,\dots ,x_n]$ the ring of polynomials in the
variables $x_1,\dots ,x_n$ with coefficients in $k$.  If $x_1^{v_1}\cdots x_n^{v_n}$ is a monomial, we denote it by $x^v$, where $v=(v_1,\dots
,v_n)\in \zn^n$.

An {\em $n$-graded ring} is a ring $R$ equipped with a decomposition $R=\bop_{v\in\zn^n} R_v$ such that $R_u\cdot R_v\subseteq R_{u+v}$. The
example we shall use is the polynomial ring $A_n=k[x_1,\dots ,x_n]$, graded by setting $A_v=kx^v$ for $v\in\zn^n$. An {\em $n$-graded module}
over an $n$-graded ring $R$ is an $R$-module $M$ with a decomposition $M=\bop_{v\in\zn^n}M_v$ such that $R_u\cdot M_v\subseteq M_{u+v}$.

Let $X$ be a topological space.  A {\em multi-filtration} of $X$ is a collection of subspaces $\{X_v\}_{v\in\zn^n}$ such that if $u\lesssim
v_1,v_2\lesssim w$, the diagram of inclusions
$$\xymatrix{
 X_{v_1}\ar[r] & X_w \\
 X_u\ar[u]\ar[r] & X_{v_2}\ar[u]}$$ commutes.  Typically, $X$ is a finite simplicial complex, in which case we assume that each $X_u$ is a
 subcomplex.  Moreover, we assume that the filtration is eventually constant in any direction; that is, if we fix all but the $i$-th coordinate
 in $\zn^n$, then for all $v\in\zn^n$ with $v_i$ sufficiently large the filtration is constant in the $i$ direction.  Also, we assume the
 filtration is finite in the sense that there is some $w\in\zn^n$ with $X_w=X$.

 Now, given a multi-filtered space $X$, we may calculate the homology of each subspace $X_v$ with coefficients in a field $k$.  The inclusion
 maps among the various subspaces yield maps on homology.  This information is encapsulated in the following definition.

 \begin{definition}\label{persistencemodule} A {\em persistence module} $M$ is a family of  $k$-vector spaces
 $\{M_v\}_{v\in\zn^n}$ together with homomorphisms $\varphi_{u,v}:M_u\to M_v$ for all $u\lesssim v$ such that if $u\lesssim v\lesssim w$ we have
 $\varphi_{v,w}\circ\varphi_{u,v} = \varphi_{u,w}$. A persistence module $M$ is {\em finite} if each $M_u$ is finite-dimensional.
 Any persistence module $M$ has the structure of an $n$-graded module over $A_n$ where the
 action of a monomial is determined by requiring $x^{v-u}:M_u\to M_v$ to be $\varphi_{u,v}$ whenever $u\lesssim v$.
 \end{definition}

 Conversely, given a finitely generated $n$-graded $A_n$-module, we get a persistence module by taking $\varphi_{u,v}:M_u\to M_v$ to be the map given
 by the action of $x^{v-u}$ on $M_u$.  Thus, there is an equivalence of categories between finite persistence modules and finitely generated
 $n$-graded $A_n$-modules.

 Note that for each $j\ge 0$, the homology modules $\{H_j(X_v)\}_{v\in\zn^n}$, together with the induced maps, yield a finite persistence
 module over $A_n$.

 Recall that if $M$ is an $n$-graded module and $v\in\zn^n$, the {\em shifted module} $M(v)$ is defined by $M(v)_u=M_{u-v}$ for all $u\in
 \zn^n$.

 \begin{definition} If $\xi$ is a multiset in $\zn^n$, the $n$-graded $k$-vector space with basis $\xi$ is the module
 $$V(\xi)=\bop_{(v,i)\in\xi} k(v).$$
 \end{definition}

 \noindent This is an $A_n$-module where the action of each variable is identically zero.

 \begin{definition} If $\xi$ is a multiset in $\zn^n$, the free $n$-graded $A_n$-module with basis $\xi$ is the module
 $$F(\xi)=\bop_{(v,i)\in\xi} k[x_1,\dots ,x_n](v).$$
 \end{definition}

 \noindent Note that each $F(\xi)_v$ is a $k$-vector space of dimension equal to $\#\{(u,i)\in\xi | u\lesssim v\}$.

 \begin{definition} If $M$ is a free $n$-graded object with basis $\xi$, we call $\xi$ the {\em type} of $M$ and denote it by $\xi(M)$.
 \end{definition}

 \subsection{Automorphisms} We now turn to automorphisms.

 \begin{definition} Let $\mu\in GL(V(\xi))$.  We say that $\mu$ {\em respects the grading} if for any $(v,i)\in\xi$, $\mu(v)$ lies in the span
 of elements $u_{ij}\in\xi$ with $u_{ij}\lesssim v$.  Denote by $GL_\lesssim(V(\xi))$ the set of all such automorphisms.
 \end{definition}

 Note that $GL_\lesssim(V(\xi))$ is an algebraic subgroup of $GL(V(\xi))$.  In fact, more is true: every element has upper triangular block
 form.  To see this, note that we may order the basis of $V(\xi)$ in the following way.  For each $(v,i)\in\xi$, order the basis elements of
 $V(\xi)_v$ arbitrarily and then order the sets according to $\lesssim$. If $v_1$ and $v_2$ are incomparable under $\lesssim$, then we order
 them arbitrarily.  For example, if $n=2$, $(0,1)$ and $(1,0)$ are incomparable, so we may choose either one to come first in the order.
  Then with respect to this ordering, any $\mu\in GL_\lesssim(V(\xi))$
 has block form
 $$\mu = \left(\begin{array}{cccc}
           L_1 & V_{12} & \cdots & V_{1r} \\
           0  & L_2  & \cdots & V_{2r} \\
           0  &  0  & \ddots & V_{r-1,r} \\
           0  &  0 & \cdots & L_r
           \end{array}\right)$$ where each $L_j\in GL(V(v_j,i_j))$ and $V_{j\ell}\in M_{i_j,i_\ell}(k)$. Note that if the degrees $v_k$ and
           $v_\ell$ are incomparable, then $V_{k\ell}=0$.

 Now, denote by $GL(F(\xi))$ the group of automorphisms of the free $n$-graded $A_n$-module $F(\xi)$ that respect the grading.  Then we have the
 following.

 \begin{proposition} The group $GL(F(\xi))$ is isomorphic to $GL_\lesssim(V(\xi))$.
 \end{proposition}

 \begin{proof} By ordering bases as above, we see that any automorphism that respects the grading of $F(\xi)$ has an upper triangular block
 form. Observe that each diagonal block has entries in $k$ since $\mu$ applied to any basis element cannot increase the grade.  Finally, note
 that, likewise, any block above the diagonal consists only of elements in $k$; for if a nonconstant polynomial is applied to a basis element,
 the grade increases. But then $\mu^{-1}$ would have to undo this action, in effect multiplying by a monomial of the form $x^{-u}$, which cannot
 happen.
 \end{proof}

 \subsection{A Family of Discrete Invariants} Suppose $M$ is a finitely-generated $n$-graded $A_n$-module.  A minimal generating set for $M$ may
 be obtained as follows.  Let $I_n$ be the ideal $(x_1,\dots ,x_n)\subset A_n$.  Let $$V(M)=M/I_nM = k\otimes_{A_n} M.$$  This is a
 finite-dimensional $n$-graded vector space, and as such it has a basis $\xi(V(M))$.  This lifts to a minimal generating set for $M$ which we
 denote by $\xi_0(M)$.

 Note that there is a canonical surjection $\varphi_M:F(\xi_0(M))\to M$.  Set $F_0=F(\xi_0(M))$.  The kernel of $\varphi_M$ is not free in
 general, but we may choose a minimal free module $F_1$ so that the sequence $F_1\to F_0\to M\to 0$ is exact.  Continuing in this way we get a
 minimal free resolution
 $$0\to F_n\to F_{n-1}\to \cdots \to F_1\to F_0\to M\to 0$$ in the category of finitely-generated $n$-graded $A_n$-modules.  That this
 resolution terminates at $F_n$ is a consequence of Hilbert's Syzygy Theorem \cite{eisenbud}, p.478.  We may use this resolution to compute Tor
 groups.  Note that for each $i\ge 0$, $\text{Tor}_i^{A_n}(M,k)$ is an $n$-graded vector space.

 \begin{definition} If $M$ is a finitely-generated $n$-graded $A_n$-module, set
 $$\xi_i(M)=\xi(\text{Tor}_i^{A_n}(M,k)).$$
 \end{definition}

 Note that the $\xi_i(M)$ are multisets in $\zn^n$ and $\xi_i(M)=\emptyset$ for $i>n$.

 \section{Relation Families and the Associated Variety}\label{variety}
 \subsection{Relation Families} Let us focus on the invariants $\xi_0(M)$ and $\xi_1(M)$ for a moment.  These correspond to a minimal generating
 set and minimal set of relations for $M$, respectively.  The following construction appears in \cite{cz}.

 \begin{definition} Let $F(\xi_0)$ and $F(\xi_1)$ be free $n$-graded $A_n$-modules.  A {\em relation family} is a collection
 $\{V_v\}_{v\in\xi_1}$ of vector spaces such that
 \begin{enumerate}
 \item $V_v\subseteq F(\xi_0)_v$;

 \item $\dim V_v = \dim F(\xi_1)_v$;

 \item if $u,v\in\xi_1$ with $u\lesssim v$, then $x^{v-u}\cdot V_u\subseteq V_v$.
 \end{enumerate}
 \end{definition}

 The collection of all such relation families is denoted by $\rf(\xi_0,\xi_1)$.

 \begin{lemma} The group $GL(F(\xi_0))$ acts on the left on $\rf(\xi_0,\xi_1)$.
 \end{lemma}

 \begin{proof} Any $\mu\in GL(F(\xi_0))$ induces an automorphism of the exact sequence
 $$F(\xi_1)\to F(\xi_0)\to M\to 0$$ and hence maps a relation family to another relation family.
 \end{proof}

 The canonical example of a relation family is given by a finitely-generated $n$-graded $A_n$-module $M$.  The map $\psi_M$ in the exact
 sequence $$F_1\stackrel{\psi_M}{\to} F_0\to M\to 0$$ gives rise to a relation family $\eta(\psi_M)$ by setting $V_v=\psi_M((F_1)_v))\subseteq
 (F_0)_v$. In \cite{cz} the authors prove the following.

 \begin{theorem} {\em (\cite{cz}, Theorem 2)} Let $\xi_0,\xi_1$ be multisets of elements from $\zn^n$ and let $[M]$ be the isomorphism class of
 finitely-generated $n$-graded $A_n$-modules $M$ with $\xi_0(M)=\xi_0$ and $\xi_1(M)=\xi_1$.  Then the assignment $[M]\mapsto \eta(\psi_M)$ is a
 bijection from the set of isomorphism classes to the set of orbits $GL(F(\xi_0))\backslash\rf(\xi_0,\xi_1)$. \hfill $\qed$
 \end{theorem}

 \subsection{The Variety Structure}\label{var}
 Note that the set $\rf(\xi_0,\xi_1)$ can be given the structure of an algebraic variety.  Indeed, given a relation family
 $\{V_v\}_{v\in\xi_1}$, each $V_v$ determines a subspace of the vector space $F(\xi_0)_v$ and so we have an inclusion of sets
 $$j:\rf(\xi_0,\xi_1)\to \prod_{(v,i)\in\xi_1} \text{Gr}_{\dim F(\xi_1)_v}(F(\xi_0)_v),$$ where $\text{Gr}_m(W)$ denotes the Grassmann variety
 of $m$-planes in $W$.

 \begin{proposition} The set $\rf(\xi_0,\xi_1)$ is a variety via the structure induced by the map $j$.
 \end{proposition}

 \begin{proof} This follows from the fact that the containment conditions $x^{v-u}\cdot V_u\subseteq V_v$ are algebraic.
 \end{proof}

 Moreover, it is clear that the action of $GL(F(\xi_0))$ on $\rf(\xi_0,\xi_1)$ is an algebraic action.  Unfortunately, the quotient space
 $GL(F(\xi_0))\backslash\rf(\xi_0,\xi_1)$ is not, in general, a variety.  In Section \ref{extend} we present one possible remedy for this.  We
 shall discuss some examples of the quotient spaces $GL(F(\xi_0))\backslash\rf(\xi_0,\xi_1)$ in Section \ref{examples}.

 \subsection{The Case $n=1$} Before proceeding, we first discuss what happens when $n=1$; that is, when we have a single filtration direction.
 In the case of a filtered space, this corresponds to the study of ordinary persistent homology \cite{zc}.  A complete classification of
 persistence modules over $k[x]$ is known---the invariants $\xi_0$ and $\xi_1$ yield a {\em barcode} showing births and deaths of homology
 classes.

 The variety $\rf(\xi_0,\xi_1)$ makes sense when $n=1$, however, and we discuss here the structure of $GL(F(\xi_0))\backslash\rf(\xi_0,\xi_1)$.

 As a simple example, suppose we are given that $\xi_0=\{0,0,2\}$ and $\xi_1=\{4\}$; that is, we have three generators appearing in filtration
 levels $0,0$, and $2$ and a single relation in level $4$.  In Figure \ref{points} we show two different filtered spaces whose $H_0$ modules
 have these $\xi_0$ and $\xi_1$.  The associated $H_0$ modules are $k[x]\oplus k[x]/x^4\oplus x^2k[x]$ and $k[x]\oplus k[x] \oplus x^2k[x]/x^4$,
 respectively.  Of course, these modules are not isomorphic, but specifying only the locations of generators and relations does not reveal this.
 Rather, the orbit space plays this role.

 \begin{figure}
 \centerline{\includegraphics[width=4.5in]{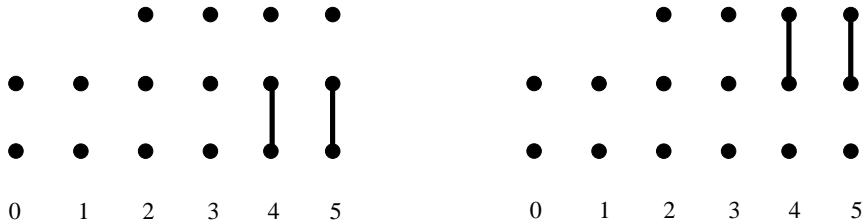}}
 \caption{\label{points}Two filtered spaces associated to $\xi_0=\{0,0,2\}$ and $\xi_1=\{4\}$}
 \end{figure}

 In this example, since there are three generators and only a single relation in a degree greater than the degrees of all the generators, we see
 that $\rf(\xi_0,\xi_1)=\text{Gr}_1(k^3)=\zp^2$.  The group $GL(F(\xi_0))$ has block form
 $$\left(\begin{array}{cc}
        GL_2(k) & V_{12} \\
        0  & GL_1(k)
        \end{array}\right).$$
 It is easy to see that $GL(F(\xi_0))$ has two orbits on $\zp^2$ and so $GL(F(\xi_0))\backslash\rf(\xi_0,\xi_1)$ consists of two points, as we
 expected.

 In general, we have the following result.

 \begin{theorem}\label{onevar} If $n=1$, the orbit space $GL(F(\xi_0))\backslash\rf(\xi_0,\xi_1)$ consists of a finite set of points.
 \end{theorem}

 \begin{proof} Since $n=1$, we are dealing with $\zn$-graded modules over the principal ideal domain $k[x]$.  The classification of modules over
 this ring tells us that in the exact sequence $$0\to F(\xi_1)\to F(\xi_0)\to M\to 0,$$ we must have $\text{rank}(F(\xi_1))\le
 \text{rank}(F(\xi_0))$.

 Let $m$ denote the number of elements in $\xi_0$ so that $GL(F(\xi_0))$ is a subgroup of $GL_m(k)$.  In this case, $GL(F(\xi_0))$ is actually a
 parabolic subgroup of $G=GL_m(k)$ since the ordering in $\zn$ is linear; that is, there are no nonzero blocks above the diagonal since all
 generator degrees are comparable.  We claim that $\rf(\xi_0,\xi_1)$ is a flag variety; that is, $\rf(\xi_0,\xi_1)=G/P$ for some parabolic
 $P\subset G$.  Order the basis elements in $\xi_1$ by degree: $e_{11},\dots ,e_{1\ell_1},e_{21}\dots ,e_{2\ell_2},\dots ,e_{r1},\dots
 ,e_{r\ell_r}$, with $\deg e_{ij}=d_i$, $j=1,\dots ,\ell_i$, and $d_1<d_2<\cdots <d_r$.  Then we have
 $$\rf(\xi_0,\xi_1)\subseteq \text{Gr}_{\ell_1}(F(\xi_0)_{d_1}) \times \text{Gr}_{\ell_1+\ell_2}(F(\xi_0)_{d_2}) \times \cdots \times
 \text{Gr}_{(\sum_{j\le r} \ell_j)}(F(\xi_0)_{d_r}).$$ Now observe that the containment conditions imply that if $(p_1,p_2,\dots ,p_r)$ lies in
 $\rf(\xi_0,\xi_1)$,
 where $\dim p_j=\sum_{i\le j}\ell_i$, then $p_1\subset p_2\subset \cdots \subset p_r$; that is, $p_1\subset p_2\subset \cdots \subset p_r$ is a
 flag in $k^m$.  It follows that we may identify $\rf(\xi_0,\xi_1)$ with the flag variety $F(d_1,d_2,\dots ,d_r)$ of flags
 $V_1\subset\cdots\subset V_r$ in $k^m$ with $\dim V_j=\sum_{j\le i} \ell_i$.
  Thus there is a parabolic subgroup $P\subset G$ with $F(d_1,\dots ,d_r)=G/P$.

 It is well-known (see e.g. \cite{borel}) that if $B$ is the group of upper triangular matrices, then the quotient $B\backslash G/P$ is finite
 (indeed, the $B$-orbits in $G/P$ give a decomposition of the projective variety $G/P$ into Schubert cells).
  Since $B\subseteq GL(F(\xi_0))$, we see that the quotient
 $$GL(F(\xi_0))\backslash\rf(\xi_0,\xi_1)\approx GL(F(\xi_0))\backslash G/P$$ is finite.
 \end{proof}

 \section{Partitions of $\rf(\xi_0,\xi_1)$ and the Associated Varieties}\label{extend}
 When $n\ge 2$, the quotient $GL(F(\xi_0))\backslash\rf(\xi_0,\xi_1)$ is often not a variety (see \cite{cz} and Section \ref{examples} below).
 In this section we present one remedy for this which has the advantage of partitioning the quotient set
 $GL(F(\xi_0))\backslash\rf(\xi_0,\xi_1)$ into a collection of varieties.  This variety structure is different from that of \cite{cz} discussed
 in Section \ref{var}, however, and does not carry quite as much information.

 \subsection{A Partition of $\rf(\xi_0,\xi_1)$}  Recall that the set $\rf(\xi_0,\xi_1)$ consists of all relation families $\{V_v\}_{v\in\xi_1}$.
  Recall further that these are in one-to-one correspondence with finitely-generated $n$-graded $A_n$-modules $M$ with $\xi_0(M)=\xi_0$ and
  $\xi_1(M)=\xi_1$.  Suppose we are given a collection $\xi_0,\xi_1,\xi_2,\dots ,\xi_n$ of multisets in $\zn^n$.

  \begin{definition} For multisets $\xi_0,\xi_1,\dots ,\xi_n$ define
  $$\rf(\xi_2,\dots ,\xi_n) = \{M\in\rf(\xi_0,\xi_1) | \xi(\text{Tor}_i^{A_n}(M,k))=\xi_i, i=2,\dots ,n\}.$$
  \end{definition}

  Note that only finitely many $\rf(\xi_2,\dots ,\xi_n)$ are nonempty.  Indeed, once $\xi_0$ and $\xi_1$ are fixed, there are only finitely many
  possibilities for locations and numbers of syzygies among the elements in $\xi_1$, showing that there are only finitely many possibilities for
  $\xi_2$. In turn there are only finitely many possibilities for $\xi_3$, and so on.  Note also that it is possible to have some
  $\xi_i=\emptyset$ (and then $\xi_j=\emptyset$ for $j\ge i$).  Moreover, the various $\rf(\xi_2,\dots ,\xi_n)$ are pairwise disjoint and their
  union is all of $\rf(\xi_0,\xi_1)$.

  \subsection{The Action of $GL(F(\xi_0))$}  It is clear that each $\rf(\xi_2,\dots ,\xi_n)$ is stable under the action of $GL(F(\xi_0))$ since
  isomorphic modules have isomorphic Tor groups.  More is true, however.

  \begin{lemma}\label{syz} Let $M$ be a finitely-generated $n$-graded $A_n$-module with minimal resolution
  $$0\to F_n\stackrel{d_n}{\to} F_{n-1}\stackrel{d_{n-1}}{\to}\cdots\to F_2\stackrel{d_2}{\to} F_1\stackrel{d_1}{\to} F_0\to M\to 0.$$ Let
  $\mu\in GL(F_0)$ and let $M'=\mu(M)$.  Then $M'$ has minimal resolution
  $$0\to F_n\stackrel{d_n}{\to} F_{n-1}\stackrel{d_{n-1}}{\to}\cdots\to F_2\stackrel{d_2}{\to} F_1\stackrel{d_1'}{\to} F_0\to M'\to 0,$$ where
  $d_1'=d_1\circ\mu$.
  \end{lemma}

  \begin{proof}This is an easy exercise and is left to the reader.
  \end{proof}

  As a consequence, we see that each isomorphism class of modules has the same {\em syzygies}, not just the same types $\xi_2,\dots ,\xi_n$.  It
  is this fact we shall exploit now.

 \subsection{A Different Variety Structure}\label{different} Given a module $M$ with a minimal free resolution
 $$0\to F_n\stackrel{d_n}{\to} F_{n-1}\stackrel{d_{n-1}}{\to}\cdots\to F_2\stackrel{d_2}{\to} F_1\stackrel{d_1}{\to} F_0\to M\to 0,$$ we may
 think of the map $d_j:(F_j)_v\to (F_{j-1})_v$, $v\in\xi_j$, as giving us a subspace of dimension $\dim (F_j)_v$ inside $(F_{j-1})_v$.  Denote
 this subspace by $d_j(M)$.  Define a variety $Y_{\xi_2,\dots ,\xi_n}$ by
 $$Y_{\xi_2,\dots ,\xi_n} = \prod_{j=2}^n\prod_{(v,i)\in\xi_j}\text{Gr}_{\dim F(\xi_j)_v}(F(\xi_{j-1})_v).$$  We may now construct a map
 $\varphi:\rf(\xi_2,\dots ,\xi_n)\to Y_{\xi_2,\dots ,\xi_n}$ by
 $$\varphi(M) = (d_2(M),d_3(M),\dots ,d_n(M)).$$  As a consequence of Lemma \ref{syz}, we obtain the following result.

 \begin{theorem} The map $\varphi$ induces a map
 $$\overline{\varphi}: GL(F(\xi_0))\backslash\rf(\xi_2,\dots ,\xi_n)\to Y_{\xi_2,\dots ,\xi_n}.$$
 \hfill $\qed$
 \end{theorem}

 Via the map $\overline{\varphi}$ we may put a variety structure on each $GL(F(\xi_0))\backslash\rf(\xi_2,\dots ,\xi_n)$ as follows. For each orbit,
 choose a representative $M$.  The module $M$ determines a point on $\rf(\xi_0,\xi_1)$ and the orbit determines a point on $Y_{\xi_2,\dots
 ,\xi_n}$.  We therefore may assemble these to obtain a map
 $$\Phi:GL(F(\xi_0))\backslash\rf(\xi_0,\xi_1)\to \coprod_{\xi_2,\dots ,\xi_n} \rf(\xi_0,\xi_1)\times Y_{\xi_2,\dots ,\xi_n}$$ defined by
 $\Phi([M]) = (M,(d_2(M),d_3(M),\dots ,d_n(M)))$.  We note that when the various $\varphi$ are injective it is not necessary to include the
 variety $\rf(\xi_0,\xi_1)$ in the definition of the map $\Phi$.  This happens, for example, if all the generators in $\xi_0$ are co-located.
 However, in Section \ref{containment} below, we give an example to show that the maps $\varphi$ need not be injective if there are generators
 in $\xi_0$ in incomparable locations.  Adding in the variety $\rf(\xi_0,\xi_1)$ forces $\Phi$ to be injective.

 The map $\Phi$ allows us to put a variety structure on $GL(F(\xi_0))\backslash\rf(\xi_0,\xi_1)$, but it is not the same as the structure of the
 quotient space defined in Section \ref{var}.  We have lost some information in our construction.  This will be discussed further in the next
 section.

\section{Examples}\label{examples}
\subsection{An Example from \cite{cz}}\label{czex}  Consider $n=2$ and the modules with generators and relations $\xi_0=\{((0,0),2)\}$ and
$\xi_1=\{((0,3),1),((1,2),1),((2,1),1),((3,0),1)\}$, respectively.  Since there are two generators which are co-located, we have
$GL(F(\xi_0))=GL_2(k)$. The set $\rf(\xi_0,\xi_1)$ is obtained by choosing, for each generator $e_v\in\xi_1$, a relation between the two
generators of $F(\xi_0)$; that is, a line in $k^2=F(\xi_0)_v$.  Since the relations are not comparable in the order $\lesssim$, there are no
containment conditions in $\rf(\xi_0,\xi_1)$, and we obtain
$$\rf(\xi_0,\xi_1) = \zp^1_k\times\zp^1_k\times\zp^1_k\times\zp^1_k.$$ The action of $GL_2(k)$ on this is the usual diagonal action.

In \cite{cz}, the authors show that the orbit space $GL_2(k)\backslash\rf(\xi_0,\xi_1)$ contains a copy of $\zp^1_k-\{0,1,\infty\}$ and
therefore that it is impossible to obtain a complete family of discrete invariants parametrizing these modules (in contrast to the $n=1$ case,
where the barcode suffices).  Let us examine this example further from the point of view of Section \ref{different}.

Since $n=2$, the only higher $\xi_i$ we need to consider is $\xi_2$.  Since elements of $\xi_2$ effectively provide syzygies among the
generators in $\xi_1$, we see that if $e\in\xi_2$, we must have $\deg e \lesssim (3,3)$ since $(3,3)$ is the least upper bound in $\zn^2$ for
the set $\{(0,3),(1,2),(2,1),(3,0)\}$.  (This follows from the fact that $A_2=k[x,y]$ is an integral domain.)  Thus, there are only six possible
locations for generators in $\xi_2$: $(1,3),(2,2),(3,1),(2,3),(3,2),(3,3)$.

To enumerate the orbits of the $GL_2(k)$-action on $(\zp^1)^4$, note that since the $GL_2(k)$-action on $\zp^1$ is $3$-transitive, we have two
cases to consider:
\begin{enumerate}
\item $\Omega=\{(\ell_1,\ell_2,\ell_3,\ell_4)\in (\zp^1)^4 | \ell_i\ne \ell_j, i\ne j\}$; and
\item those $4$-tuples where at least two of the $\ell_i$ are the same.
\end{enumerate}
The second case admits further refinement as well.

Consider the set $\Omega$.  Since $GL_2(k)$ acts $3$-transitively on $\zp^1$, we see immediately that
 \begin{eqnarray*}
 GL_2(k)\backslash\Omega & = &  \{(0,\infty,1,\ell) | \ell\in\zp^1-\{0,1,\infty\}\} \\
           & \cong & \zp^1-\{0,1,\infty\}.
 \end{eqnarray*}
 Let $\alpha\in k-\{0,1\}$ correspond to the line $\ell$.  Then the relations of the associated isomorphism class of modules are
 $$[y^3,0], [0,xy^2], [-x^2y,x^2y], [-x^3,\alpha x^3];$$ that is, the first generator dies at $(0,3)$, the second at $(1,2)$, the generators
 become equal at $(2,1)$ and $\alpha$ times the second equals the first at $(3,0)$.  An easy calculation shows that there are two syzygies among
 these:
 \begin{center}
 $x^2[y^3,0] - xy[0,xy^2] + y^2[-x^2y,x^2y]$\\
 $(1-\alpha)[0,xy^2] - xy[-x^2y,x^2y] + y^2[-x^3,\alpha x^3]$
 \end{center}
 in degrees $(2,3)$ and $(3,2)$, respectively.  It follows that $\Omega\subset \rf(\{(2,3),(3,2)\}$.  The space $Y_{\xi_2}$ for
 $\xi_2=\{((2,3),1),((3,2),1)\}$ is $\text{Gr}_1(k^3)\times \text{Gr}_1(k^3) = \zp^2\times\zp^2$ and the map
 $\overline{\varphi}:GL_2(k)\backslash\Omega\to Y_{\xi_2}$ is given by
 $$\overline{\varphi}((0,\infty,1,\ell)) = ([1:-1:1],[1-\alpha:-1:1]),$$ as revealed by the syzygies above.  Note that the image of
 $\overline{\varphi}$ is a copy of $\zp^1-\{0,1,\infty\}$ inside $\zp^2\times\zp^2$, embedded in the second factor.

 There are other elements in $\rf(\xi_2)$ for $\xi_2=\{((2,3),1),((3,2),1)\}$ besides those in $\Omega$, and these fall in the category of those
 points in $(\zp^1)^4$ having fewer than $4$ distinct coordinates.  Since $GL_2(k)$ acts $3$-transitively on $\zp^1$, we see that dividing up
 points in $(\zp^1)^4$ into groups by numbers and locations of distinct points, that there are finitely many $GL_2(k)$ orbits in
 $(\zp^1)^4-\Omega$.  These are enumerated in Table \ref{ex1table}.  The orbit representative column indicates which elements of a $4$-tuple are
 the same.  For example, $(0,\infty,0,1)$ means that the first and third coordinates of any element in the orbit are equal.  The $\xi_2$ column
 shows the degrees of generators for the syzygies among the relations in $\xi_1$.  For each $\xi_2$, we have the variety $Y_{\xi_2}$ defined
 above, and finally $\text{im}\overline{\varphi}$ indicates the point on $Y_{\xi_2}$ mapped to by the particular orbit.  In this example, the
 various $\overline{\varphi}$ are injective, and so we may ignore the $\rf{\xi_0,\xi_1}$ portion.

 \begin{table}
 \begin{tabular}{c|c|c|c}
 orbit rep & $\xi_2$ & $Y_{\xi_2}$ & $\text{im}\overline{\varphi}$ \\ \hline\hline
 $(0,1,\infty,\alpha)$ & $(2,3),(3,2)$ & $\zp^2\times\zp^2$ & $([1:-1:1],[1-\alpha:-1:1])$ \\
 $(0,0,\infty,1)$ & $(1,3),(3,2)$ & $\zp^1\times\zp^2$ & $([1:-1],[1:-1:1])$ \\
 $(0,\infty,0,1)$ & $(2,3),(3,2)$ & $\zp^2\times\zp^2$ & $([1:0:-1],[1:-1:-1])$ \\
 $(0,\infty,1,0)$ & $(2,3),(3,3)$ & $\zp^2\times\zp^3$ & $([1:-1:1],[1:0:0:-1])$ \\
 $(0,\infty,\infty,1)$ & $(2,2),(3,3)$ & $\zp^1\times \zp^3$ & $([1:-1],[1:0:-1:1])$ \\
 $(0,\infty,1,\infty)$ & $(2,3),(3,2)$ & $\zp^2\times\zp^2$ & $([1:-1:1],[1:0:-1])$ \\
 $(0,\infty,1,1)$ & $(2,3),(3,1)$ & $\zp^2\times\zp^1$ & $([1:1:-1],[1:-1])$ \\
 $(0,0,\infty,\infty)$ & $(1,3),(3,1)$ & $\zp^1\times\zp^1$ & $([1:-1],[1:-1])$ \\
 $(0,\infty,0,\infty)$ & $(2,3),(3,2)$ & $\zp^2\times\zp^2$ & $([1:0:-1],[1:0-1])$ \\
 $(0,\infty,\infty,0)$ & $(2,2),(3,3)$ & $\zp^1\times\zp^3$ & $([1:-1],[1:0:0:-1])$ \\
 $(0,0,0,\infty)$ & $(1,3),(2,2)$ & $\zp^1\times\zp^1$ & $([1:-1],[1:-1])$ \\
 $(0,0,\infty,0)$ & $(1,3),(3,2)$ & $\zp^1\times\zp^2$ & $([1:-1],[1:0:-1])$ \\
 $(0,\infty,0,0)$ & $(2,3),(3,1)$ & $\zp^2\times\zp^1$ & $([1:0:-1],[1:-1])$ \\
 $(\infty,0,0,0)$ & $(2,2),(3,1)$ & $\zp^1\times\zp^1$ & $([1:-1],[1:-1])$ \\
 $(0,0,0,0)$ & $(1,3),(2,2),(3,1)$ & $\zp^1\times\zp^1\times\zp^1$ & $([1:-1],[1:-1],[1:-1])$
 \end{tabular}
 \caption{\label{ex1table} The $GL_2(k)$ orbits}
 \end{table}

 Observe that there are $9$ distinct $\xi_2$'s and that via the map
 $$\Phi:GL_2(k)\backslash\rf(\xi_0,\xi_1)\to \coprod_{\xi_2} Y_{\xi_2}$$ we may put a variety structure on the quotient set
 $GL_2(k)\backslash\rf(\xi_0,\xi_1)$.  Of course, this is not the quotient of $\rf(\xi_0,\xi_1)$ in the category of varieties and we have lost a
 great deal of information in the process.  For example, consider $\xi_2=\{(2,3),(3,2)\}$.  This includes the generic set $\Omega$ along with
 the hyperplanes $H_{13}=\{(\ell_1,\ell_2,\ell_1,\ell_4)\}$, $H_{24}=\{(\ell_1,\ell_2,\ell_3,\ell_2)\}$, and their intersection $H_{13}\cap
 H_{24}$.  The map $\Phi$, restricted to $GL_2(k)\backslash\rf(\xi_2)$ has image
 $$\{([1:-1:1],[1-\alpha:-1:1])|\alpha\ne 0,1\} \cup {} $$ $$\{([1:0:-1],[1:-1:-1]),([1:-1:1],[1:0:-1]),([1:0:-1],[1:0:-1])\}.$$
 Note that the three points are separated from $GL_2(k)\backslash\Omega=\zp^1-\{0,1,\infty\}$.  However, if one uses the variety structure on
 $\rf(\xi_0,\xi_1)$ and considers the induced quotient topology on $GL_2(k)\backslash\rf(\xi_0,\xi_1)$, the three points are in the closure of
 $GL_2(k)\backslash\Omega$.  Thus, with our approach, information about degeneracies is lost at the expense of obtaining a description of the
 collection of isomorphism classes of modules as a variety.

 \subsection{An Example with Containment Conditions}\label{containment}
 Consider the following multisets:
 \begin{eqnarray*}
 \xi_0 & = & \{((0,1),1),((1,0),2)\} \\
 \xi_1 & = & \{((2,0),1),((1,1),1),((1,2),1)\}.
 \end{eqnarray*}
 Then we have
 $$\rf(\xi_0,\xi_1) = \{(\ell_1,\ell_2,p)\in \zp^1_{(2,0)}\times\zp^2_{(1,1)}\times\text{Gr}_2(k^3)_{(1,2)}|\ell_2\subset p\}$$ (here, the
 subscript on each factor indicates the degree of the relation).  The group $GL(F(\xi_0))$ is isomorphic to $GL_1(k)\times GL_2(k)$ and if we
 order the generators by letting the first lie at $(0,1)$, then $GL(F(\xi_0))$ has block form
 $$\left(\begin{array}{cc}
      GL_1(k) & 0 \\
        0 & GL_2(k)
        \end{array}\right).$$

 The orbits of the action on $\rf(\xi_0,\xi_1)$ are shown in Table \ref{table2}. An element of $\text{Gr}_2(k^3)$ is denoted $\{ab\}$,
 where $a$ and $b$ are coordinate directions or sums of such (e.g., $x$ or $x+z$). The final orbit in the table is isomorphic to $\zp^1-\{0,\infty\}$.

  Notice that the map $\overline{\varphi}$ is not injective.
 This happens because there are nonisomorphic modules with the same higher syzygies.  This is a consequence of having generators in different,
 incomparable locations.  A typical such occurrence is when the same relation is imposed among the generators in degree $(1,0)$ at locations
 $(2,0)$ and $(1,1)$, yielding a syzygy at $(2,1)$.  We may then choose several inequivalent relations at $(1,2)$ that yield no higher syzygies.
 That is why we need to use the variety $\rf{\xi_0,\xi_1}$ to distinguish orbits in this case.

 \begin{table}
 \begin{tabular}{c|c|c|c|c}
 orbit rep $\omega$ & $\xi_2$ & $Y_{\xi_2}$ & $\text{im}\overline{\varphi}$ & $\text{im}\Phi$ \\ \hline\hline
 $([1:0],[0:1:0],\{yz\})$ & $(2,1)$ & $\zp^1$ & $[1:-1]$ & $(\omega,[1:-1])$ \\
 $([1:0],[0:1:0],\{xy\})$ & $(2,1)$ & $\zp^1$ & $[1:-1]$ & $(\omega,[1:-1])$ \\
 $([1:0],[0:1:0],\{x+z,y\})$ & $(2,1)$ & $\zp^1$ & $[1:-1]$ & $(\omega,[1:-1])$ \\
 $([1:0],[0:0:1],\{xz\})$ & $\emptyset$ & $\ast$ & $\ast$ & $(\omega,\ast)$ \\
 $([1:0],[0:0:1],\{x+y,z\})$ & $\emptyset$ & $\ast$ & $\ast$ & $(\omega,\ast)$ \\
 $([1:0],[0:0:1],\{yz\})$ & $(2,2)$ & $\zp^2$ & $[1:0:-1]$ & $(\omega,[1:0:-1])$ \\
 $([1:0],[1:0:0],\{xy\})$ & $(2,2)$ & $\zp^2$ & $[1:0:-1]$ & $(\omega,[1:0:-1])$ \\
 $([1:0],[1:0:0],\{xz\})$ & $\emptyset$ & $\ast$ & $\ast$ & $(\omega,\ast)$ \\
 $([1:0],[1:1:0],\{xy\})$ & $(2,2)$ & $\zp^2$ & $[1:1:-1]$ & $(\omega,[1:1:-1])$ \\
 $([1:0],[1:1:0],\{x+y,x+z\})$ & $\emptyset$ & $\ast$ & $\ast$ & $(\omega,\ast)$ \\
 $([1:0],[1:0:1],\{xz\})$ & $\emptyset$ & $\ast$ & $\ast$ & $(\omega,\ast)$ \\
 $([1:0],[1:0:1],\{x+z,y\})$ & $(2,2)$ & $\zp^2$ & $[1:0:-1]$ & $(\omega,[1:0:-1])$ \\
 $([1:0],[1:0:1],\{x+z,y+tz\}),t\ne 0$ & $\emptyset$ & $\ast$ & $\ast$ & $(\omega,\ast)$
 \end{tabular}
 \caption{\label{table2} Orbits of the $GL(F(\xi_0))$ action}
 \end{table}

 \section{Geometric Interpretation of $\xi_i$, $i\ge 2$}\label{geometric}
 Consider the filtration of the circle shown in Figure \ref{circlefilt}.  The $H_0$ and $H_1$ modules for this filtration are shown in Figure
 \ref{h0h1}.  For $H_0$, we have $\xi_0=\{((0,0),3)\}$, $\xi_1=\{((0,1),1),((1,0),1),((2,1),1)\}$, and $\xi_2=\{((2,1),1)\}$, while $H_1$ is
 free with $\xi_0=\{((2,1),1)\}$.  Note the connection between $\xi_2$ for $H_0$ and $\xi_0$ for $H_1$.  This occurs in almost every example one
 writes down.  In this section, we provide an explanation for this.

 \begin{figure}
 \centerline{\includegraphics[height=2in]{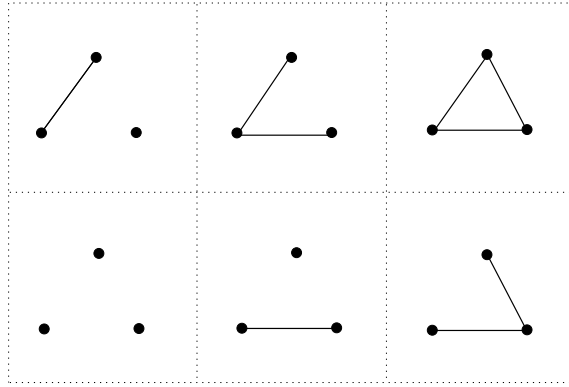}}
 \caption{\label{circlefilt} A filtration of the circle}
 \end{figure}

 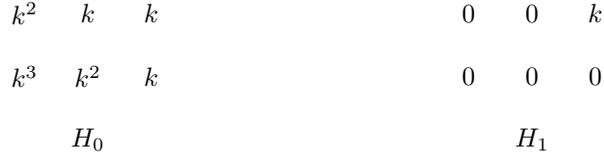
\begin{figure}
 $$\xymatrix{
 k^2  &  k  &  k  &  &  &  &  & 0  & 0 & k \\
 k^3  & k^2 & k   &  &  &  &  & 0 & 0 & 0 \\ \hline
     & H_0 &     &  &  &  &  &   & H_1 & }$$
 \caption{\label{h0h1}The modules $H_0$ and $H_1$}
 \end{figure}

 \subsection{Chain Complexes of $n$-graded Modules}
 Suppose that $X$ is a finite simplicial complex filtered by $\zn^n$; denote the filtered space by $X_\bullet$.  For each $i\ge 0$, we have the
 $n$-graded $A_n$-module $C_i(X_\bullet) = \bop_{v\in\zn^n} C_i(X_v)$ of $i$-chains.  Since the boundary map is functorial with respect to maps
 of simplicial complexes, we obtain a chain complex in the category of $n$-graded $A_n$-modules:
 $$C_\bullet(X_\bullet) = \{\cdots\to C_i(X_\bullet)\stackrel{\partial}{\to} C_{i-1}(X_\bullet)\to\cdots \}.$$  The $i$-th homology of this
 complex is the $n$-graded module $H_i(X_\bullet) = \bop_{v\in\zn^n} H_i(X_v)$.

 In previous sections, we studied these homology modules individually.  This point of view, however, allows us to draw connections between them
 as hinted at by the example above.  For this we need the hypertor groups of the complex $C_\bullet(X_\bullet)$.  A reference for this material
 is Section 5.7 of \cite{weibel}.

 The hypertor modules of $C_\bullet(X_\bullet)$ are the Tor groups in the category of $n$-graded $A_n$-modules:
 $$\tor_p^{A_n}(C_\bullet(X_\bullet),M)$$ where $M$ is any $n$-graded $A_n$-module.  Here, we are interested only in $M=k$, sitting in degree
 $(0,\dots ,0)\in\zn^n$.  These are obtained by taking a Cartan--Eilenberg resolution $P_{\bullet\bullet}\to C_\bullet(X_\bullet)$
 (see p.~145 of \cite{weibel}) or $K_{\bullet\bullet}\to k$, and taking the homology of the total complex of the resulting double complex
 obtained by tensoring two objects:
 $$P_{\bullet\bullet}\otimes_{A_n} k \quad \text{or}\quad C_\bullet(X_\bullet)\otimes_{A_n} K_{\bullet\bullet} \quad \text{or}\quad
 P_{\bullet\bullet}\otimes_{A_n} K_{\bullet\bullet}.$$ As usual, there are two spectral sequences for computing the homology of the total
 complex.  One of these has (by taking horizontal homology first)
 $$E_{pq}^2 = \text{Tor}_p^{A_n}(H_q(X_\bullet),k) \Rightarrow \tor_{p+q}^{A_n}(C_\bullet(X_\bullet),k).$$

 We shall discuss the abutment in the next section.  For now, let us focus on the $E^2$-term itself.  Set $p=2$ and note that we have a map
 $$d^2_{2q}:\text{Tor}_2^{A_n}(H_q(X_\bullet),k) \to \text{Tor}_0^{A_n}(H_{q+1}(X_\bullet),k).$$  In other words, we have a functorial way to
 relate elements of $\xi_2(H_q(X_\bullet))$ to elements of $\xi_0(H_{q+1}(X_\bullet))$.

 In the circle example above, we have a map of graded modules
 $$d^2_{20}:\text{Tor}_2^{A_2}(H_0(X_\bullet),k)\to \text{Tor}_0^{A_2}(H_1(X_\bullet),k).$$  As $H_1(X_\bullet)$ is a free $A_2$-module with
 generator at $(2,1)$, we have the resolution
 $$0\to A_2(2,1) \stackrel{\cong}{\to} H_1(X_\bullet)\to 0$$ and so $\text{Tor}_i^{A_2}(H_1(X_\bullet),k)$ is the $i$-th homology of the complex
 $$0\to A_2(2,1)\otimes_{A_2} k \to 0.$$ Thus, we get only $\text{Tor}_0$ and it is a single copy of $k$ in degree $(2,1)$.

 To compute the Tor groups of $H_0(X_\bullet)$, we use the resolution
 $$0\to A_2(2,1)\stackrel{d_2}{\to} A_2(1,0)\oplus A_2(0,1)\oplus A_2(2,0) \stackrel{d_1}{\to} A_2(0,0)^3 \to H_0(X_\bullet)\to 0,$$
 where the maps $d_i$ are given by
 $$d_2 = \left(\begin{array}{c}
           -xy \\ x^2 \\ -y
           \end{array}\right) \qquad d_1 = \left(\begin{array}{ccc}
                                              x & y & 0 \\
                                              -x & 0 & x^2 \\
                                              0 & -y & -x^2
                                              \end{array}\right).$$
 Applying $-\otimes_{A_2} k$ we obtain the complex
 $$0\to k(2,1)\stackrel{0}{\to} k(1,0)\oplus k(0,1)\oplus k(2,0) \stackrel{0}{\to} k(0,0)^3\to 0,$$ from which we deduce that
 $\text{Tor}_2^{A_2}(H_0(X_\bullet),k)=k(2,1)$.  The map $d^2_{20}$ is then a map $k(2,1)\to k(2,1)$.

 The easy way to see that this map is an isomorphism is to note that since $C_0(X_\bullet)$ and $C_1(X_\bullet)$ are free $A_2$-modules, we have
 $\tor_i^{A_2}(C_\bullet(X_\bullet),k) = H_i(C_\bullet(X_\bullet)\otimes_{A_2} k)$.  It is easy to see that
 \begin{eqnarray*}
 \tor_0^{A_2}(C_\bullet(X_\bullet),k) & = & k(0,0)^3  \\
 \tor_1^{A_2}(C_\bullet(X_\bullet),k) & = & k(0,1) \oplus k(1,0) \oplus k(2,0),
 \end{eqnarray*}
 and these modules occur in the $E^2$-term of the spectral sequence in degrees $(0,0)$ and $(1,0)$, respectively.  Since $E^3=E^\infty$ in this
 case, we must have that $d^2_{20}:k(2,1)\to k(2,1)$ is an isomorphism.

 In fact, this map is $-\text{id}$, as can be seen by choosing a Cartan--Eilenberg resolution of the complex $C_\bullet(X_\bullet)$.  Upon
 applying the functor $-\otimes_{A_2} k$ to this resolution, we find that the $E^0$-term of the (transposed) spectral sequence is
 $$\xymatrix{
 {\begin{array}{c}
 k(1,0)\oplus k(0,1) \\
 \oplus \\
 k(2,0)\oplus k(2,1)
 \end{array}}  & k(2,1) & 0 \\
 {\begin{array}{c}
 k(0,0)^3\oplus k(1,0) \\
 \oplus \\
 k(0,1)\oplus k(2,0)
 \end{array}}  & {\begin{array}{c}
                   k(1,0)\oplus k(0,1) \\
                   \oplus \\
                   k(2,0)\oplus k(2,1)
                   \end{array}} & k(2,1) \\
  }$$ where
 the horizontal and vertical maps are either $0$ or $\pm\text{id}$ as allowed by grading (horizontal maps are $\text{id}$,
 vertical maps $-\text{id}$ in column
 $1$ and $\text{id}$ in column $0$).  It follows that the $E^1$-term is
 $$\xymatrix{
 k(2,1) & 0 & 0 \\
 k(0,0)^3 & k(1,0)\oplus k(0,1)\oplus k(2,0) & k(2,1)}$$ and hence $E^1=E^2$.  Given the description of the $E^0$-term, it is now clear,
 by construction, that $d^2_{20}:E^2_{20}\to E^2_{01}$ is $-\text{id}$.

 Geometrically, this may be interpreted as follows.  The generator of $\text{Tor}_2^{A_2}(H_0(X_\bullet),k)$ represents the first location where
 a collection of relations in $H_0$ is not an independent set.  This can occur due to duplications of relations, or, as in this example, because
 these have come together to form a $1$-cycle.  This is a general result.

 \begin{theorem}\label{tor2tor0} The kernel of the map
 $$d^2_{2q}:\text{\em Tor}_2^{A_n}(H_q(X_\bullet),k)\to \text{\em Tor}_0^{A_n}(H_{q+1}(X_\bullet),k)$$ is generated by syzygies resulting from the
 same relation being imposed in $H_q(X_\bullet)$ in multiple degrees.  If a nonzero $w\in\text{\em Tor}_0^{A_n}(H_{q+1}(X_\bullet),k)$ is in the
 image of $d^2_{2q}$, say $d^2_{2q}z=w$, then $w=\sum\alpha_i w_i$ for some $(q+1)$-simplices $w_i$ where each $w_i$ corresponds to an element
 of $\text{\em Tor}_1^{A_n}(B_q(X_\bullet),k)$ and $z$ gives a syzygy among the $w_i$.
 \end{theorem}

 \begin{proof} If one constructs a Cartan--Eilenberg resolution of $C_\bullet(X_\bullet)$ as in \cite{weibel}, p.~146, one discovers that the
 differential $d^2_{2q}$ is built as follows.  For each $q$, we have  exact sequences of $A_n$-modules
 $$0\to B_q(X_\bullet) \to Z_q(X_\bullet)\to H_q(X_\bullet)\to 0,$$
 and $$0\to Z_{q+1}(X_\bullet)\to C_{q+1}(X_\bullet)\to B_q(X_\bullet)\to 0,$$ together with the associated long exact sequence of Tor groups.
 We may piece these together as follows:
 $$\text{Tor}_2^{A_n}(H_q,k)\stackrel{f}{\to}\text{Tor}_1^{A_n}(B_q,k)\to\text{Tor}_0^{A_n}(Z_{q+1},k)\to
 \text{Tor}_0^{A_n}(H_{q+1},k).$$ Denote the composite of the last two maps by $g$.  Then $d^2_{2q}=g\circ f$.

 Denote the minimal resolutions of $H_q(X_\bullet)$ and $B_q(X_\bullet)$ by $P_{q\bullet}\to H_q$ and $Q_{q\bullet}\to B_q$. The differentials in these
 resolutions will be denoted $d_P$ and $d_Q$, respectively.  Recall that $I$
 denotes the ideal $(x_1,\dots ,x_n)\subset A_n$.
 Suppose $d^2_{2q}z=0$, $z\ne 0$.  We have two cases.

 \begin{enumerate}
 \item[1.] $f(z)=0$.  Then we have $f(z)\in I\cdot Q_{q1}$; that is, $f(z)=\sum x^vz_v$, where $z_v\in \xi(Q_{q1})$.  It is easy to see
 that $d_Q\circ f=0$ and so $d_Q(\sum x^vz_v) = 0$; that is, $\sum x^v d_Qz_v = 0$.  This element is trivial in $\text{Tor}_1^{A_n}(B_q,k)$.  It
 therefore gives an inessential relation among the generators of $B_q$.  Let $u_v=d_qz_v$.  Then $u_v\in Q_{q0}$ and $\sum x^vu_v=0$.  The $u_v$
 correspond to $(q+1)$-chains and the relation $\sum x^vu_v=0$ shows that the  $u_v$ are a redundant set of relations; i.e., we have imposed the
 same relations $(u_v)$ among $q$-chains in different grades.  Thus, the syzygy $z$ arises from imposing these redundant relations.

 \item[2.] $f(z)\ne 0$.  Then $f(z)=\sum z_v + \sum x^v w_v$ for some $z_v,w_v\in \xi(Q_{q1})$, and we have $g(f(z))=\sum x^v u_v$ for some
 $u_v\in P_{q+1,0}$.  Again, this element is trivial, this time in $\text{Tor}_0^{A_n}(H_{q+1}(X_\bullet),k)$, and so it must be that we have
 imposed the same relations $(u_v)$ among $q$-simplices in different grades.
 \end{enumerate}

 Finally, if $d^2_{2q}(z)=w\ne 0$, then $w=\sum u_v + \sum x^v w_v$ for some $u_v,w_v\in\xi(P_{q+1,0})$.  Write $f(z)=\sum y_v + \sum x^v z_v$,
 $y_z,w_z\in\xi(Q_{q1})$.  Note that $\sum y_v$ represents a nontrivial element of $\text{Tor}_1^{A_n}(B_q(X_\bullet),k)$, and the
 $(q+1)$-dimensional homology class $w$ is built from $q$-boundaries $y_v$ with $z$ giving a syzygy among them.
 \end{proof}

 In the case $n=2$, this gives the complete picture since $E^3=E^\infty$.  For $n\ge 3$, however, we have higher differentials in the spectral
 sequence.  In particular, the differential
 $d^\ell_{\ell q}:E^\ell_{\ell q}\to E^\ell_{0,q+\ell-1}$ is a map from a subquotient of $\text{Tor}_{\ell}^{A_n}(H_q(X_\bullet),k)$ to a
 quotient of $\text{Tor}_0^{A_n}(H_{q+\ell-1}(X_\bullet),k)$.  This is a much more subtle and complicated relationship that we shall not
 investigate here.

 \subsection{The Abutment}\label{abut}
 In this section we investigate the groups $\tor_i^{A_n}(C_\bullet(X_\bullet),k)$ that form the abutment of the spectral sequence investigated
 above.  Recall the filtration of the circle shown in Figure \ref{circlefilt} above.  We had
 \begin{eqnarray*}
 \tor_0^{A_n}(C_\bullet(X_\bullet),k) & = & k(0,0)^3 \\
 \tor_1^{A_n}(C_\bullet(X_\bullet),k) & = & k(0,1)\oplus k(1,0)\oplus k(2,0).
 \end{eqnarray*}
 Note that each of these groups is generated by elements that correspond to actual simplices in the space $X$; the three points for $\tor_0$ and
 the three edges for $\tor_1$.  As it stands, the grading prevents the existence of a nontrivial map between these groups, but if we drop the
 grading, there is an obvious map $\partial:\tor_1\to\tor_0$ given by the geometric boundary.  The homology of this complex is then the homology
 of the underlying space.

 This example is particularly nice since the chain groups $C_i(X_\bullet)$ are free $A_2$-modules.  Still, we shall now show that in certain cases
 it is possible
 to construct a boundary map $\partial_i:\tor_i(C_\bullet(X_\bullet),k)\to\tor_{i-1}(C_\bullet(X_\bullet),k)$, after forgetting the grading, so
 that the homology of the resulting complex is the homology of the space $X$.

 There is another spectral sequence converging to the modules $\tor_i(C_\bullet(X_\bullet),k)$; its $E^1$-term satisfies
 $$E^1_{pq} = \text{Tor}_q^{A_n}(C_p(X_\bullet),k)$$ and the $q$-th row of the $E^2$-term is obtained by taking the homology of the complex
 $E^1_{\bullet q}=\{\text{Tor}_q(C_\bullet(X_\bullet),k),d^1\}$, where $d^1$ is induced by the boundary map in $C_\bullet(X_\bullet)$.  The
 bottom row, $E^1_{\bullet 0}$ is simply the complex $C_\bullet(X_\bullet)\otimes_{A_n} k$.  Each $C_i(X_\bullet)\otimes_{A_n} k$ is an
 $n$-graded $k$-vector space with basis the $i$-simplices in $X$, but possibly with duplications if a particular cell enters the filtration in
 different grades.

 When $n=1$, there can be no such duplications; it therefore follows that if we ignore the grading and use the geometric boundary map, the
 bottom row is just $C_\bullet(X)\otimes k$.  Note also that when $n=1$, the modules $C_i(X_\bullet)$ are necessarily free (because of the lack
 of duplications) and so $\text{Tor}_1^{k[x]}(C_i(X_\bullet),k)=0$ for all $i\ge 0$.  It follows that
 $\tor_i(C_\bullet(X_\bullet),k)=C_i(X_\bullet)\otimes_{k[x]} k$, and with the boundary map $\partial_i:\tor_i\to\tor_{i-1}$, we see that
 $$H_i(\tor_\bullet^{k[x]}(C_\bullet(X_\bullet),k),\partial_\bullet) \cong H_i(X;k).$$

 For $n\ge 2$, things are more complicated.  Note that the $d^1$-map is mostly zero for degree reasons.  For example, on the bottom row, the
 only way an $i$-simplex $\sigma$ might map to something nonzero is if $\sigma$ and some part of $\partial\sigma$ entered the filtration in the
 same degree.  Otherwise, there is nowhere for $\sigma$ to go.  In the result below, we shall get around this by assuming that at most one
 simplex is added at a time as we move in any direction in the filtration.

 We first note the following fact.

 \begin{proposition} For all $j\ge n+\dim X$, we have $\tor_j^{A_n}(C_\bullet(X_\bullet),k) = 0$.
 \end{proposition}

 \begin{proof} For $j>n+\dim X$, this follows by noting that in the spectral sequence above, we have $E^1_{pq}=0$ for $p>\dim X$ or $q>n$.
 More is true, however.  We claim that $\text{Tor}_n(C_p(X_\bullet),k)=0$ as well, and this implies that $\tor_{n+\dim
 X}^{A_n}(C_\bullet(X_\bullet),k)=0$.  To see this, let $K_\bullet\to k$ be the Koszul complex, where each term sits in degree $(0,0,\dots
 ,0)\in\zn^n$.  This is a free resolution of $k$ in the category of $n$-graded $A_n$-modules, and so we may use it to compute
 $\text{Tor}_q^{A_n}(C_p(X_\bullet),k)$.  In this way, it is easy to see that
 $$\text{Tor}_n^{A_n}(C_p(X_\bullet),k) = \{\sigma\in C_p(X_\bullet): x_j\sigma=0, j=1,\dots ,n\},$$ and this is clearly zero since chains never
 die (in contrast to homology classes, where $\text{Tor}_n$ could be nonzero).
 \end{proof}

 Heuristically speaking, an element of $\text{Tor}_q^{A_n}(C_p(X_\bullet),k)$, $q>0$ may be thought of as a virtual $(p+q)$-cell in the following
 way.  If $q=1$, we get elements that tell us to identify two $p$-simplices in the space $X$ that are copies of the same cell in different
 filtration levels.  We may view this element as a $(p+1)$-cell that fills in the void created by attaching two copies of the $p$-simplex.  An
 example of this is shown in Figure \ref{spherefig}, where a $2$-cell has been added in filtration levels $(1,2)$ and $(3,0)$ (note that this is a
 cellular filtration, not a simplicial one, but the principles are the same).  The
 corresponding relation gives rise to an element of $\tor_3^{A_2}(C_\bullet(X_\bullet),k)$ in degree $(3,2)$ (coming from an element of
 $\text{Tor}_1^{A_2}(C_2(X_\bullet),k)$) which we can picture as a $3$-cell filling in the sphere created by attaching two copies of the
 $2$-cell to a cylinder.  Elements in $\text{Tor}_q$ may be thought of similarly---higher syzygies arise from syzygies appearing in multiple places,
 and these get filled in by virtual cells.

 \begin{figure}
 \centerline{\includegraphics[height=2.5in]{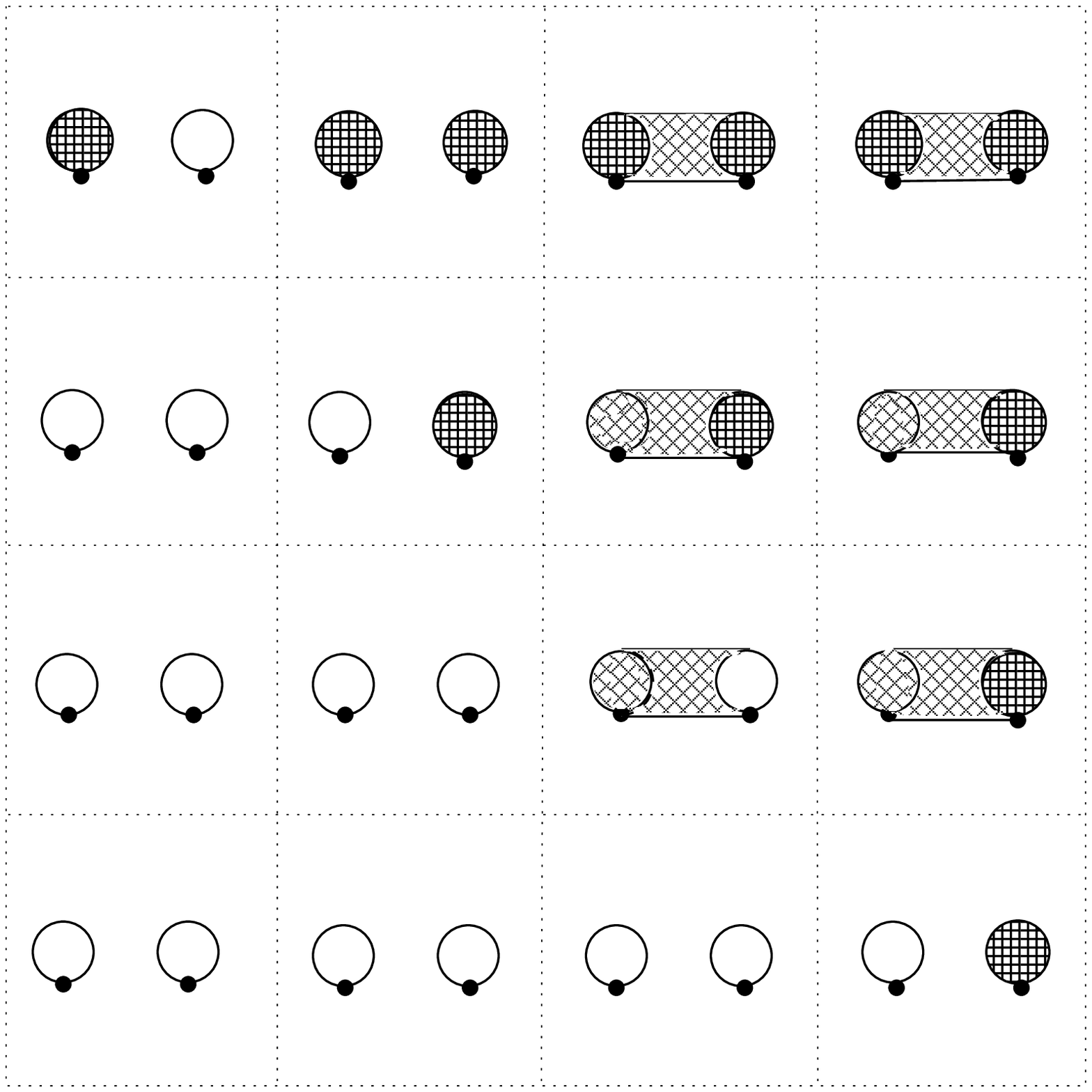}}
 \caption{\label{spherefig} A filtered sphere with duplicated cells in degrees $(1,2)$ and $(3,0)$.  The horizontal edge in the bottom of
 the picture in degree $(2,1)$ is a $1$-cell joining the two vertices. The $2$-cell attached in degree $(2,1)$ creates a
 cylinder with boundary the two circles.}
 \end{figure}

 We may therefore think of the elements of $\tor_j^{A_n}(C_\bullet(X_\bullet),k)$ as $j$-cells, some of which correspond to real $j$-cells in
 the space $X$ (if they come from $\text{Tor}_0^{A_n}(C_j(X_\bullet),k)$), others of which correspond to virtual cells filling in spheres
 created by duplications.  We now define a map $\partial:\tor_\ell^{A_n}(C_\bullet(X_\bullet),k)\to\tor_{\ell-1}^{A_n}(C_\bullet(X_\bullet),k)$,
 ignoring the grading.
 Note that $\tor_\ell=\bop_{i+j=\ell} E_{ij}^\infty$, and that each $E_{ij}^\infty$ is a subquotient of $\text{Tor}_j^{A_n}(C_i(X_\bullet),k)$.
 If $j>0$, we view an element of $E_{ij}^\infty$ as a virtual $(i+j)$-cell giving a relation among elements of $E_{i,j-1}^1$; say $z\in
 E_{ij}^1$ yields a syzygy among elements $z_i$ in $E_{i,j-1}^1$.  We then define $\partial[z]=\sum [z_i]$.  Note that if we do not ignore the
 degrees of the elements of these modules, this would often be the zero map.  However, ignoring the grading, this makes sense. For $j=0$, we
 have elements of $E_{i,0}^\infty$ coming from actual cells in the space $X$.  For such a simplex $\sigma$, define $\partial [\sigma]= [\partial
 \sigma]$, the geometric boundary of $\sigma$.

 While a general description of the $\tor_j$ is unwieldy, we do have the following result. Denote by $e_j$ the element
 $(0,0,\dots ,0,1,0,\dots ,0)\in\zn^n$, where the $1$ is in the $j$-th position.

 \begin{theorem}\label{hypertor} Suppose the filtration $X_\bullet$ is such that for every $v\in\zn^n$, $$\dim_k((C_i(X_\bullet)_{v+e_j}))\le
 1+\dim_k((C_i(X_\bullet)_v),$$ $j=1,\dots ,n$ (that is, we add at most one simplex at a time moving in any coordinate direction).  Then for all
 $\ell\ge 0$, $$\tor_\ell^{A_n}(C_\bullet(X_\bullet),k) = \bop_{i+j=\ell} \text{\em Tor}_j^{A_n}(C_i(X_\bullet),k),$$ and using
 $\partial:\tor_\ell(C_\bullet(X_\bullet),k)\to\tor_{\ell-1}(C_\bullet(X_\bullet),k)$ defined above, we have
 $$H_\bullet(\tor_\bullet(C_\bullet(X_\bullet),k),\partial) \cong H_\bullet(X;k).$$
 \end{theorem}

 \begin{proof} Note that the condition on the filtration implies that in the $E^0$-term of the spectral sequence, the horizontal differential is
 identically zero for degree reasons (i.e., the boundary of any simplex lives in a lower degree, a relation involves objects in lower degrees,
 etc., and so after tensoring with $k$, everything vanishes).  It follows that $E^1=E^\infty$ and so for $\ell\ge 0$,
 $$\tor_\ell^{A_n}(C_\bullet(X_\bullet),k) = \bop_{i+j=\ell}E_{ij}^\infty = \bop_{i+j=\ell}\text{Tor}_j^{A_n}(C_i(X_\bullet),k).$$

 Denote the group $\tor_\ell^{A_n}(C_\bullet(X_\bullet),k)$, with grading dropped, by $T_\ell$.  We have an inclusion of chain complexes
 $$\varphi:C_\bullet(X;k)\to T_\bullet$$ defined as follows.  Let $\sigma$ be a generator of $C_i(X;k)$.  Then
 $\sigma\in\text{Tor}_0^{A_n}(C_i(X_\bullet),k)$, perhaps in multiple locations.  Choose the copy in degree $v=(v_1,\dots ,v_n)$, where
 $v_1,\dots ,v_n$ are minimized in order.  For example, if $n=3$ and $\sigma$ occurs in grades $(0,1,2)$ and $(0,2,1)$, we would choose the
 generator in grade $(0,1,2)$ for $\varphi(\sigma)$.  Since we have dropped the grading in $T_i$, this choice is unimportant and gives a
 well-defined chain map.  Denote by $Q_\bullet$ the quotient complex so that we have an exact sequence
 $$0\to C_\bullet(X;k)\stackrel{\varphi}{\to} T_\bullet\to Q_\bullet \to 0.$$  We claim that $H_\bullet(Q_\bullet)\equiv 0$.

 Note that for each $\ell\ge 0$,
 $$Q_\ell = \bop_{i+j=\ell,j\ge 1} \text{Tor}_j^{A_n}(C_i(X_\bullet),k) \oplus  \text{Tor}_0^{A_n}(C_\ell(X_\bullet),k)/\varphi(C_\ell(X;k)),$$
 and hence $Q_\bullet$ is the total complex of the double complex obtained from $E^\infty$ by taking the quotient of the bottom row by the image
 of $C_\bullet(X;k)$ under $\varphi$.  We claim that the vertical homology of this double complex vanishes so that $H_\bullet(Q_\bullet)\equiv
 0$.  To see this, note that the map $$E_{i,1}^\infty\to E_{i,0}^\infty/\varphi(C_i(X;k))$$ is surjective since the elements in
 $\text{Tor}_1^{A_n}(C_i(X_\bullet),k)$ serve to identify duplications of simplices in $X$.  Since we have set one of these simplices equal to
 zero, the others get hit by the appropriate element of $\text{Tor}_1^{A_n}(C_i(X_\bullet),k)$.  Similarly, a syzygy $z$ in $E_{ij}^\infty$
 corresponds to a virtual $(i+j)$-cell that fills in the $2$ $(i+j-1)$-cells it relates.  It follows that any vertical cycle may be filled with
 a vertical boundary and hence $E^1\equiv 0$, as required.
 \end{proof}

 \subsection{Examples}
 \noindent 1.  Consider the filtered circle shown in Figure \ref{oneatatime}.  The chain groups of this space are:
 \begin{eqnarray*}
 C_0(X_\bullet) & = & \frac{A_2(0,0)\oplus A_2(0,1)\oplus A_2(1,0)\oplus A_2(2,1)\oplus
 A_2(4,0)}{x(0,1,0,0,0)=y(0,0,1,0,0),x^2(0,0,0,1,0)=y(0,0,0,0,1)} \\
 C_1(X_\bullet) & = & \frac{A_2(1,1)\oplus A_2(2,0)\oplus A_2(3,2)\oplus A_2(4,1)\oplus A_2(4,2)}{x(1,0,0,0,0)=y(0,1,0,0,0),
 x(0,0,1,0,0)=y(0,0,0,1,0)}.
 \end{eqnarray*}
 The $E^1=E^\infty$-term of the spectral sequence for computing $\tor_\bullet^{A_n}(C_\bullet(X_\bullet),k)$ is then
 $$\small{\xymatrix{
 k(1,1)\oplus k(4,1) & k(2,1)\oplus k(4,2) \\
 k(0,0)\oplus k(0,1)\oplus k(1,0)\oplus k(2,1)\oplus k(4,0) & k(1,1)\oplus k(2,0)\oplus k(3,2)\oplus k(4,1)\oplus k(4,2).}}$$
 We therefore have $T_0=E^1_{0,0}$, $T_1 = E^1_{0,1}\oplus E^1_{1,0}$, and $T_2=E^1_{1,1}$, and the complex $T_\bullet$ is then
 $$k^2 \stackrel{B}{\to} k^2\oplus k^5 \stackrel{A}{\to} k^5$$ where the matrices $A$ and $B$ are
 $$A=\left(\begin{array}{rrrrrrr}
      0 & 0 & 1 & 1 & 0 & 0 & 1 \\
      1 & 0 & 0 & 0 & 0 & 0 & 0 \\
      -1 & 0 & -1 & -1 & -1 & -1 & 0 \\
      0 & 1 & 0 & 0 & 1 & 1 & 0 \\
      0 & -1 & 0 & 0 & 0 & 0 & -1
      \end{array}\right) \quad B=\left(\begin{array}{rr}
                                  0 & 0 \\
                                  0 & 0 \\
                                  1 & 0 \\
                                  -1 & 0 \\
                                  0 & 1 \\
                                  0 & -1 \\
                                  0 & 0
                                  \end{array}\right).$$
 It is easy to see that $A$ has rank $4$, $B$ has rank $2$, and so the homology of this complex is $k$ in degrees $0$ and $1$, and $0$ in degree
 $2$.  The inclusion $\varphi:C_\bullet(X;k)\to T_\bullet$ takes the three vertices to those in degrees $(0,0)$, $(0,1)$, and $(2,1)$, and the
 three edges to those in degrees $(1,1)$, $(3,2)$, and $(4,2)$.  The map $\varphi$ is a quasi-isomorphism.

 \begin{figure}
 \centerline{\includegraphics[height=2.5in]{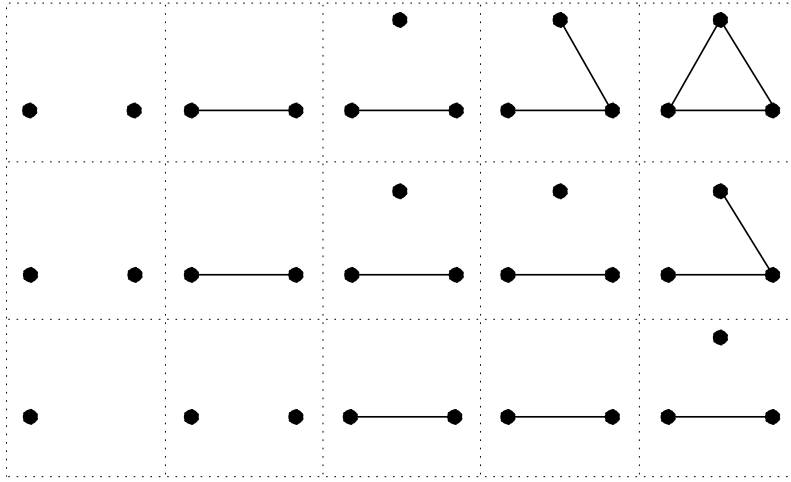}}
 \caption{\label{oneatatime} A circle with one simplex entering at a time}
 \end{figure}

 \medskip

 \noindent 2. Observe that the filtered sphere shown in Figure \ref{spherefig} fails the criterion imposed in the statement of the theorem since we add
 multiple cells when passing to degree $(2,1)$.  However, we may still recover the homology of $X=S^2$ from the hypertor groups.
 The chain groups of this filtered space are as follows:
 \begin{eqnarray*}
 C_0(X_\bullet) & = & A_2(0,0)^2 \\
 C_1(X_\bullet) & = & A_2(0,0)^2 \oplus A_2(2,1) \\
 C_2(X_\bullet) & = & \frac{A_2(0,3)\oplus A_2(1,2)\oplus A_2(2,1)\oplus A_2(3,0)}{x^2(0,1,0,0)=y^2(0,0,0,1)}.
 \end{eqnarray*}
 The spectral sequence for calculating $\tor_\bullet^{A_n}(C_\bullet(X_\bullet),k)$ has $E^1$-term
 $$\xymatrix{
 0  &  0  &  k(3,2) \\
 k(0,0)^2 & k(0,0)^2\oplus k(2,1)\ar[l]_<<<<0 & k(0,3)\oplus k(1,2)\oplus k(2,1)\oplus k(3,0)\ar[l]_<<<<0}$$
 Since the horizontal differential is zero, this is the $E^\infty$-term as well.  Note that we have isomorphisms
 $\varphi:C_0(X;k)\to T_0$ and $\varphi:C_1(X;k)\to T_1$, while the inclusion $\varphi:C_2(X;k)\to T_2$ is given by
 \begin{eqnarray*}
 \sigma_1 & \mapsto & \sigma_1(0,3) \\
 \sigma_2 & \mapsto & \sigma_2(1,2) \\
 \tau & \mapsto & \tau(2,1)
 \end{eqnarray*}
 where $\sigma_1$ is the $2$-cell entering in degree $(0,3)$, $\sigma_2$ enters at $(1,2)$ and $\tau$ enters at $(2,1)$ to create the cylinder.
 For clarity, we have indicated the degree of each element in $T_i$, but we have dropped the grading for calculations. We therefore have
 $Q_0=0$, $Q_1=0$, $Q_2=k(3,0)$, and $Q_3=k(3,2)$.  The map $\partial:Q_3\to Q_2$ is the identity; the group $Q_3$ is generated by a virtual
 $3$-cell filling in the two copies of $\sigma_2$.  Note that $H_\bullet(Q_\bullet)\equiv 0$, and so $\varphi:C_\bullet(X;k)\to T_\bullet$ is a
 quasi-isomorphism.

\end{document}